\documentclass{article}
\usepackage[utf8]{inputenc}
\usepackage{amsmath,amsthm,amssymb}
\usepackage{graphicx}
\usepackage{color}

\title{Planar tropical gravitational descendants with one tangency condition}
\author{Falko Gauss}

\date{\today}

\theoremstyle{plain}
\newtheorem{theorem}{Theorem}[section]
\newtheorem{lemma}[theorem]{Lemma}

\newtheorem{corollary}[theorem]{Corollary}

\theoremstyle{definition}
\newtheorem{definition}[theorem]{Definition}
\newtheorem{assumption}[theorem]{Assumption}

\theoremstyle{remark}
\newtheorem{remark}[theorem]{Remark}

\newcommand{\tref}[1]{Theorem~\textup{\ref{#1}}}

\newcommand{\cref}[1]{Corollary~\textup{\ref{#1}}}
\newcommand{\lref}[1]{Lemma~\textup{\ref{#1}}}

\newcommand{\fref}[1]{Figure~\textup{\ref{#1}}}

\begin{document}

\maketitle

\smallskip
\noindent \textbf{Abstract:}
We prove a formula which allows us to recursively compute planar tropical gravitational descendants which involve psi-classes of arbitrary power at marked ends fixed by points and additionally a psi-class of power one at exactly one marked end fixed by a line. Up to now almost exclusively tropical gravitational descendants with psi-classes at marked ends restricted by points had been considered.



\section{Introduction}
\label{introduction}

Gromov-Witten theory is an important tool to tackle enumerative problems in algebraic geometry. The main objects of Gromov-Witten theory are the so-called Gromov-Witten invariants. They count (virtually) the number of curves of given genus and degree that pass through a certain number of chosen subvarieties.
To also count curves that fulfill some tangency conditions, geometers started to use gravitational descendants which are generalized versions of Gromov-Witten invariants. Note that some authors also call these invariants descendant Gromov-Witten invariants (e.g. \cite{markrau}).\\
With the emergence of tropical geometry also tropical counterparts of Gromov-Witten invariants and gravitational descendants were introduced. Unfortunately, the tropical gravitational descendants do not coincide with the corresponding classical numbers in general.
This difference has structural reasons. The moduli space of marked tropical curves in $\mathbb{R}^{2}$ is non-compact and, hence, does not parameterize curves with components that approach the "boundary" of $\mathbb{R}^{2}$ in form of a limiting process.
In particular, this fact makes the computation, which relies classically on the recursive structure of the boundary of the respective moduli space, much harder in the tropical setting.
Nevertheless, some results on the computation of certain tropical gravitational descendants have been achieved.
Most notably, Johannes Rau and Hannah Markwig proved a tropical analogue to the WDVV Equation for planar invariants where psi-classes appear only at marked ends restricted by codimension two objects (see \cite{markrau}). Here we want to go one step further and prove a Topological Recursion Relation which allows us to compute invariants which can additionally have a psi-class of power one at exactly one marked end restricted by a codimension one object.
In order to do so we have to recall some notions from \cite{gathkerbmark} and \cite{raunew}. This will be done in Section \ref{definitions} of the article at hand. Lastly, in Section \ref{mainsection} we will prove the desired formula using methods from \cite{markrau} and \cite{raunew}.\\ For the sake of completeness, one should mention that there is very recent work of Peter Overholser which deals also with the computation of tropical descendant invariants in the plane (see \cite{overholser}). But his point of view is more mirror symmetric and less intersection theoretic.\\
I would like to thank Andreas Gathmann for several helpful discussions. Note that this article contains material from my master thesis.\\

\section{Defining tropical gravitational descendants}
\label{definitions}

We consider the \emph{moduli space $\mathcal{M}^{lab}_{0,l+m+n,trop}\left(\mathbb{R}^{2},d\right)$ of rational parameterized tropical curves} in $\mathbb{R}^{2}$ of degree $\Delta$ with $l+m+n+\left|\Delta\right|$ labelled ends of which $l+m+n$ are contracted (see \cite[Definition 4.1]{gathkerbmark} for details).
In analogy to the moduli space of pointed stable curves the contracted ends are called \emph{marked ends}. In particular, we assume here that the degree is of the form $\Delta=\{\underbrace{\left(-e_{1},-e_{2},e_{1}+e_{2}\right),\ldots,\left(-e_{1},-e_{2},e_{1}+e_{2}\right)}_{d \times}\}$, where $e_{1}$ and $e_{2}$ denote the unit vectors in the plane $\mathbb{R}^{2}$. By abuse of notation we will not distinguish between the labels of non-contracted ends and the corresponding elements in $\Delta$. Furthermore, we will refer to $-e_{1}$, $-e_{2}$ and $e_{1}+e_{2}$ as the \emph{standard directions}.\\
Exemplarily, an element $\left(\Gamma,x_{1},\ldots,x_{5},h\right)$ from $\mathcal{M}^{lab}_{0,5,trop}\left(\mathbb{R}^{2},2\right)$, i.e. a \emph{marked tropical curve}, is shown in \fref{markedcurve}.

\begin{figure}[h]
  \centering
	\def\svgwidth{260pt}
  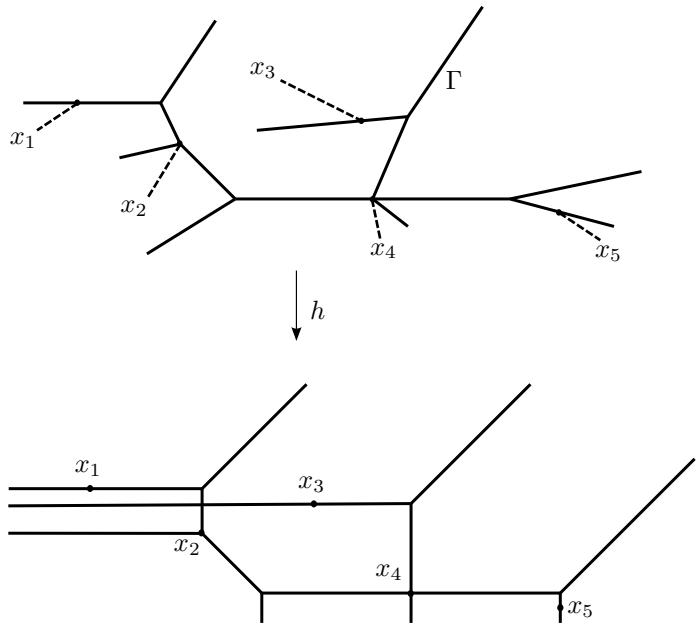
  \caption{5-marked tropical curve}
  \label{markedcurve}
\end{figure}

According to what one sees $\Gamma$ is an \emph{abstract tropical curve}, i.e. a tree with a distinguished number of marked ends whose bounded edges are equipped with strictly positive lengths. The marked ends of the abstract tropical curve are visualized as dashed lines. The map $h:\Gamma \rightarrow \mathbb{R}^{2}$ maps $\Gamma$ to an image as on the bottom of \fref{markedcurve}. The image satisfies the balancing condition and marked ends are contracted to a point.\\

Now, we want to count marked tropical curves that fulfill certain incidence and tangency restrictions. 
Since we are working in the plane, we only have to consider incidence restrictions coming from codimension one objects, i.e. lines, and codimension two objects, i.e. points. The codimension zero object, i.e. the whole space $\mathbb{R}^{2}$, does not give us non-trivial restrictions. So let $G_{j}$ be certain (tropical) lines and $P_{k}$ be certain points in $\mathbb{R}^{2}$. For us a tropical line is a 3-valent graph with one vertex and one end in each of the standard directions (cf. the left-most curve in \cite[Definition 2.7]{markrau}). We will call the unique vertex of the tropical line its \emph{root}.\\
We can partition the set of marked ends of a curve into three disjoint sets. In the following we will consider curves with $l+m+n$ marked ends and a partition $L\cup M \cup N=\left\{1,\ldots,l+m+n\right\}$ with $\left|L\right|=l$, $\left|M\right|=m$ and $\left|N\right|=n$. We assume that the ends with labels $i \in L$ are unrestricted, the ends with labels $j \in M$ have to meet the respective lines $G_{j}$ and the ends with labels $k \in N$ have to meet the respective points $P_{k}$. To make the meaning of the word "meet" precise, we need the following definition.

\begin{definition}
\label{Definition:evaluation}
For $i=1,\ldots,l+m+n$ the map $$ev_{i}:\mathcal{M}^{lab}_{0,l+m+n,trop}\left(\mathbb{R}^{2},d\right)\rightarrow \mathbb{R}^{2},\left(\Gamma,x_{1},\ldots,x_{l+m+n},h\right) \mapsto h\left(x_{i}\right)$$ is called the \emph{$i$-th evaluation map}.
\end{definition}

By \cite[Proposition 4.7]{gathkerbmark} the moduli space of rational parameterized tropical curves can be equipped with a tropical (=weighted, balanced and polyhedral) fan structure via choosing one of the marked ends as the \emph{anchor vertex}.
So the evaluation maps become morphisms of fans. Moreover, in this sense we can talk about \emph{cones}, \emph{facets}, \emph{(relative) interiors} and \emph{rational functions} (cf. \cite[Section 2]{markrau}). Lastly, this enables us to apply the tropical intersection theory developed in \cite{allerau} and to form \emph{intersection products} consisting of rational functions on the fan
$\mathcal{M}^{lab}_{0,l+m+n,trop}\left(\mathbb{R}^{2},d\right)$. These intersection products have the structure of tropical complexes and, hence, we can talk about \emph{polyhedra} and the \emph{dimension} of an intersection product.

Now we want to enrich these intersection products further by allowing \emph{psi-classes} $\psi_{i}$ at the marked ends with indices $i \in L\cup M \cup N$ (see \cite[Definition 2.2]{markrau} for the definition of a psi-class in tropical geometry). They can be thought of as encoding tangency (or higher order contacts) of the curves at the respective marked ends. In formal terms a psi-class $\psi_{i}$ is the codimension one subfan of $\mathcal{M}^{lab}_{0,l+m+n,trop}\left(\mathbb{R}^{2},d\right)$ that consists of cones corresponding to trees where the marked end $i$ sits at a vertex of valence $\geq 4$. 
Let $t_{i}$, $s_{j}$ and $r_{k}$ be the numbers that specify how many psi-classes at $i \in L$, $j \in M$ and $k \in N$, respectively, are required.\\

\begin{definition}
[cf. {\cite[Proposition and Definition 2.10]{markrau}}]
\label{defgeneral}
Let $d,l,m,n,t_{i},s_{j}$ and $r_{k}$ be non-negative integers which satisfy:
\begin{align}
\label{dimensionequ}
l+m+n+3d-3+2=m+2n+\sum_{i \in L} t_{i}+\sum_{j \in M} s_{j}+\sum_{k\in N} r_{k}.
\end{align}
Then the respective \emph{tropical gravitational descendant} is defined as:
$$\left\langle \prod_{i\in L} \tau_{t_{i}}\left(0\right)\prod_{j \in M}\tau_{s_{j}}\left(1\right)\prod_{k \in N}\tau_{r_{k}}\left(2\right)\right\rangle_{d}:=$$
$$\frac{1}{\left(d!\right)^{3}}\deg\left(\prod_{i\in L} \psi^{t_{i}}_{i}\prod_{j\in M} ev^{*}_{j}\left(G_{j}\right)\psi^{s_{j}}_{j}\prod_{k\in N} ev^{*}_{k}\left(P_{k}\right)\psi^{r_{k}}_{k}\cdot\mathcal{M}^{lab}_{0,l+m+n,trop}\left(\mathbb{R}^{2},d\right)\right).$$
\end{definition}

Note that the dimension of $\mathcal{M}^{lab}_{0,l+m+n,trop}\left(\mathbb{R}^{2},d\right)$ is $l+m+n+3d-3+2$ by \cite[Proposition 4.7]{gathkerbmark} and the codimension of the intersection of psi-classes is $\sum_{i \in L} t_{i}+\sum_{j \in M} s_{j}+\sum_{k\in N} r_{k}$ where the pullback of a line has codimension 1 and the pullback of a point codimension 2.
So the lengthy Equation \ref{dimensionequ} enforces the intersection product to be zero-dimensional. This assumption is necessary to talk about the \emph{degree} of an intersection product.
The definition above is independent of the choice of the lines $G_{j}$ and the points $P_{k}$ (cf. Section 2 of \cite{markrau}).
In particular, the definition also makes sense if we choose rather special configurations of lines and points.
But to give some enumerative meaning to the definition we have to assume a certain generality of the lines $G_{j}$ and the points $P_{k}$.
Namely, we assume that the lines $G_{j}$ and the points $P_{k}$ are in \emph{general position} in the sense of \cite[Definition 3.2]{gathkerbmark}. 
It is worth mentioning that the assumption that $G_{j}$ and $P_{k}$ are in general position is not a serious limitation because the $P_{k}$ and $G_{j}$ in general position form a dense subset in the set of $P_{k}$ and $G_{j}$ under all conditions (which can be identified with some large-dimensional real space). A proof of this fact is presented in \cite[Lemma 3.4]{markrau}.\\
Assuming that the lines $G_{j}$ and points $P_{k}$ are in general position, the intersection product consists (in set-theoretic terms) of the tropical curves that pass through $G_{j}$ and $P_{k}$ equipped with certain weights. The degree of the intersection product, to wit the tropical gravitational descendant, is then the sum of all weights. According to \cite[Lemma 2.4]{franz} (or \cite[Remark 2.2.14]{rau}) these weights can be computed as the product of the weight of the facet the respective curve lies in and some determinantal expression.
Nevertheless, the computation of the weights via determinantal formulas is very unwieldy and unsuitable for a large degree or a large number of marked ends.
That is why one wish for recursion formulas, like for instance the Topological Recursion Relation in classical Gromov-Witten theory. Such formulas rely classically on the well-known Splitting Lemma (e.g. see \cite[5.2.1 Lemma]{kock}).
Johannes Rau developed a tropical version of the Splitting Lemma.
However, his Tropical Splitting Lemma \cite[Theorem 4.13]{raunew} only works under rather strict assumptions.
In fact, this is one of the major problems when transfering the classical Gromov-Witten theory to the tropical setting. An example where the Tropical Splitting Lemma fails can be found in \cite[Remark 4.18]{raunew}. Here we restrict ourselves to a special situation.

\begin{assumption}
\label{notation}
From now on (unless stated otherwise) we assume that $L=\emptyset$, $M=\{1\}$, $\{2,3\} \subseteq N$ and $s_{1}=1$.
\end{assumption}

So we have a fan $X:= \psi^{0}_{1}\prod_{k\in N} \psi^{r_{k}}_{k} \cdot \mathcal{M}^{lab}_{0,1+n,trop}\left(\mathbb{R}^{2},d\right)$ and a one-dimensional intersection product $F:=\left( \tau_{0}\left(1\right)\prod_{k\in N}\tau_{r_{k}}\left(2\right)\right)$.
This situation has two major advantages. 
Firstly, the weight $w_{X}\left(\sigma\right)$ of the facet $\sigma$ the respective curve lies in is one by \cite[Theorem 5.3]{markrau} and, therefore, can be omitted in this article.
Secondly, we can apply the Tropical Splitting Lemma by \cite[Lemma 4.16]{raunew}. The Tropical Splitting Lemma splits the curves of an intersection product at their \textit{contracted bounded edges} as described in \cite[Remark 2.7]{gathmark}. Note that a contracted bounded edge induces a \emph{reducible partition} $I|J$ of the labels of the marked ends of the curve, i.e. a partition into two non-empty sets $I$ and $J$ such that directions in $I\cap \Delta$ (and hence also in $J\cap \Delta$) sum up to the zero vector. In our situation the Tropical Splitting Lemma \cite[Theorem 4.13]{raunew} takes the form of:
\begin{align}
\label{eqhardterms2}
&\notag\left\langle\tau_{1}\left(1\right) \prod_{k\in N}\tau_{r_{k}}\left(2\right)\right\rangle_{d}=\\
&\notag\sum_{\substack{I|J \textrm{ reducible }\\1 \in I, 2,3 \in J}} \sum_{\varepsilon+\zeta=2}\left\langle \tau_{0}\left(\varepsilon\right)\tau_{0}\left(1\right)\prod_{k\in N \cap I}\tau_{r_{k}}\left(2\right)\right\rangle_{d_{1}}\cdot\left\langle \tau_{0}\left(\zeta\right)\prod_{k\in N\cap J}\tau_{r_{k}}\left(2\right)\right\rangle_{d_{2}}\\
&+  \left\langle \phi_{1|2,3} \tau_{0}\left(1\right) \prod_{k\in N}\tau_{r_{k}}\left(2\right)\right\rangle_{d}.
\end{align}

In Equation \eqref{eqhardterms2} the rational function $\phi_{1|2,3}$ is defined as specified in \cite[Section 5.2]{raunew}. Here, it is sufficient for us to know that $\phi_{1|2,3}$ is of the following form (see \cite[Lemma 5.6]{raunew}) on each facet $\sigma$ of $F$, i.e. on each top-dimensional polyhedron of $F$:
\begin{align}
\label{phi}
\phi_{1|2,3}|_{\sigma}=\begin{cases} \textrm{constant},&  \textrm{\it{if the interior curves of $\sigma$ have}}\\
&\textrm{\it{a contracted bounded edge}}\\ 
\textrm{sum of lengths of edges that}\\
\textrm{separate $1$ from $2,3$},&  \textrm{\it{otherwise.}} \end{cases}
\end{align}

Loosely speaking $\phi_{1|2,3}$ decides if the result applying the Tropical Splitting Lemma equals the result one obtains by applying the classical Splitting Lemma to the respective classical invariants. This is the case in almost all of the existing literature (e.g. in \cite{markrau} and \cite{raunew}), because there $\phi_{1|2,3}$ is bounded.
Here $\phi_{1|2,3}$ shows a slightly different behaviour. This will be made precise in the next corollary.

\begin{definition}
\label{movablestring}
Let $C$ be a curve in $\mathcal{M}^{lab}_{0,1+n,trop}\left(\mathbb{R}^{2},d\right)$. A subgraph $\mathcal{S}$ of $C$ that is homeomorphic to $\mathbb{R}$ and does not involve (the closure of) a marked end $x_{k}$ with $k \in N$ is called a \emph{string}.
\end{definition}

\begin{corollary}
\label{lemmanewdefor}
Let $F$ be as above.
Then, the function $\phi_{1|2,3}$ can be unbounded only on facets of $F$ with interior curves that look like in \fref{specialexample} (where $E$ and $G_{1}$ can run along any of the three standard directions).
\begin{figure}[h]
  \centering
		\def\svgwidth{280pt}
  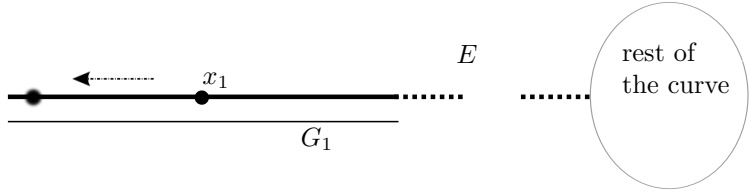
  \caption{Deformation which yields an unbounded $\phi_{1|2,3}$}
  \label{specialexample}
\end{figure}
\end{corollary}

\begin{proof}
Every facet of $F$ is described by the deformation of one of its interior curves, because facets are one-dimensional.
A list of all possible deformations is presented in \cite[Lemma 5.13]{raunew}. In our situation ($L=\emptyset$) the list looks as follows.
The deformation of an interior curve $C$ in a facet $\sigma$ of $F$ is described by one of the following cases:
\begin{enumerate}
\item $C$ contains a contracted bounded edge and one can vary its length
arbitrarily.
\item $C$ does not have a contracted bounded edge; and
\begin{enumerate}
\item $C$ has a \emph{degenerated} vertex $V$ of valence three, i.e. a vertex of one of the following types:
\begin{enumerate}
\item The vertex is incident to a bounded edge $E$, a non-contracted end and the marked end with label $1$ such that the curve $C$ and the line
$G_{1}$ do not intersect transversally at $ev_{1}\left(C\right)$.
\item The vertex is incident to a bounded edge $E$ and two non-contracted ends such that the span of the non-contracted ends near
$V$ is still only one-dimensional.
\end{enumerate}
The deformation is given by changing the length of the bounded edge $E$.

\item $C$ contains a string $\mathcal{S}$ with two non-contracted ends and only one adjacent
bounded edge $E$ as shown in \fref{specialexample2} such that all vertices of $\mathcal{S}$ are non-degenerated and of valence three in $C$. This string can be moved by changing the length of $E$.
\end{enumerate}
\end{enumerate}

Now we will work through this list and stick to the notations from there.\\
Assume that $\sigma$ is a facet of $F$ which has an interior curve with a contracted bounded edge, i.e. $\sigma$ looks like 1. on the list. Then $\phi_{1|2,3}$ is constant on $\sigma$ by \eqref{phi} and, thus, bounded on $\sigma$.\\
Now, assume that $\sigma$ is a facet with an interior curve as in 2. (a) on the list.
As $\phi_{1|2,3}$ measures the sum of the lengths of the edges that separate $1$ from $2,3$, the function $\phi_{1|2,3}$ can only be unbounded on $\sigma$ if it takes the length of the growing edge $E$ into account. Note that, in the case 2. (a) ii. $E$ does not separate $1\in M$ from $2, 3$ and, hence, does not contribute to $\phi_{1|2,3}$ on $\sigma$. This holds, because the vertex $V$ in 2. (a) ii. is of valence three and, hence, cannot also be incident to the marked end with label $1$. In the case 2. (a) i. the only chance that $E$ is taken into account by $\phi_{1|2,3}$ is the situation shown in \fref{specialexample}.\\
Finally, assume that $\sigma$ is a facet as in 2. (b) on the list. In this case an interior curve of $\sigma$ looks like in \fref{specialexample2}. Namely, $E$ only contributes to $\phi_{1|2,3}$ on 
$\sigma$, if $E$ separates $1$ from $2,3$. So by construction $x_{1}$ has to lie on $\mathcal{S}$ (cf. \cite[Lemma 2.4.14]{rau}) and $x_{2}, x_{3}$ have to lie "behind" $E$. Moreover, on such a curve $E$ must have one of the standard directions (i.e. directions of the edges of $G_{1}$), because otherwise the edge $E$ will not grow infinitely for obvious reasons.
\begin{figure}[h]
  \centering
	 \def\svgwidth{250pt}
  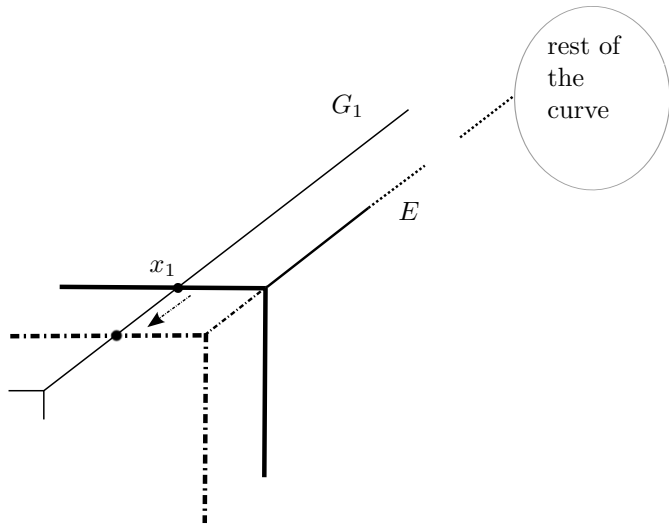
  \caption{Bounded deformation of a string}
  \label{specialexample2}
\end{figure}
But the line $G_{1}$ was chosen to look like the left-most curve in \cite[Definition 2.7]{markrau}. Hence, it has a root and none of the half-rays can be infinitely long in both directions. In \fref{specialexample2} this means that the line $G_{1}$ on the left hand side has to split somewhere, i.e. the root has to appear. So $E$ cannot grow infinitely and the situation in \fref{specialexample2} cannot cause an unbounded $\phi_{1|2,3}$ either.\\
All in all, the only interior curves which can cause facets with unbounded $\phi_{1|2,3}$ are the ones in \fref{specialexample}.
\end{proof}

\section{A new Topological Recursion Relation}
\label{mainsection}

A very helpful technique to compute certain intersection products, resp. their degrees, is (what we will call) the \emph{$\lambda$-method}.
The $\lambda$-method is due to Andreas Gathmann and Hannah Markwig (see \cite{gathmark}).
To apply the $\lambda$-method one adds a factor $ft^{*}\left(\lambda\right)$ to a one-dimensional intersection product of interest. This factor is the pull-back of an element $\lambda \in \mathcal{M}_{0,4,trop}$ under the forgetful map $ft:\mathcal{M}^{lab}_{0,l+m+n,trop}\left(\mathbb{R}^{2},d\right)\rightarrow\mathcal{M}_{0,4,trop}$ which forgets all ends except those with the labels $1,2,3,4$ of a given curve. Here $\mathcal{M}_{0,4,trop}$ is the \emph{moduli space of 4-marked abstract tropical curves} as defined in \cite[Definition 3.2]{gathkerbmark}. This moduli space consists of four different types of curves. All different types are shown in \cite[Example 1.1]{gathmark}. In the following we omit the case that $\lambda$ is just the single vertex in $\mathcal{M}_{0,4,trop}$, i.e. looks like the right-most curve in \cite[Example 1.1]{gathmark}.
Then, by abuse of notation, $\lambda$ denotes not only an element in $\mathcal{M}_{0,4,trop}$ but also a real number. This real number is the length of the unique bounded edge (cf. \cite[Example 1.1]{gathmark}). Along these lines, the forgetful map $ft$ can be considered as a map to the set of real numbers.\\
Curves in an intersection product that involves $ft^{*}\left(\lambda\right)$ fulfill the additional requirement that they map to $\lambda$ under $ft$.
Now, the key idea of the $\lambda$-method is to make $\lambda$ sufficiently large. Then, the curves in the intersection product have certain special properties, but the degree of the intersection product remains the same by the so-called rational equivalence (cf. the proof of \cite[Theorem 8.1]{markrau}).\\

The overall goal of this section is to recursively compute invariants of the form $\left\langle\tau_{0}\left(0\right)^{l'}\tau_{1}\left(0\right)^{l''} \tau_{0}\left(1\right)^{m-1}\tau_{1}\left(1\right)\prod_{k\in N}\tau_{r_{k}}\left(2\right)\right\rangle_{d}$ where $l',l'' \in \mathbb{N}$ and $m\in \mathbb{N}_{>0}$.
We will see that in order to do so, it is sufficient to derive a recursion formula for invariants of the form $\left\langle \tau_{1}\left(1\right)\prod_{k\in N}\tau_{r_{k}}\left(2\right)\right\rangle$, i.e. invariants as they appeared in Section \ref{definitions}. As we learned in the last section the curves shown in \fref{specialexample} are of particular importance.
The next lemma is the essential step to describe these curves in the realm of the $\lambda$-method.  

\begin{lemma}
\label{corollaryextrapoint}
Let $F$ be as in the previous section and $\sigma$ be a facet of $F$. Moreover, let $C$ be an interior curve of $\sigma$ as shown in \fref{specialexample}, such that $x_{4}$ with $4\in \Delta$ is the unbounded edge adjacent to the growing bounded edge $E$ and $C$ maps to $\lambda$ under $ft$. Then the linear parts $\left(\phi_{1|2,3}\right)_{\sigma}$ and $\left(ft\right)_{\sigma}$ of the restricted functions $\phi_{1|2,3}|_{\sigma}$ and $ft|_{\sigma}$ coincide.
\end{lemma}

\begin{proof}
We know from \cite[Lemma 5.13]{raunew} that we get other curves of $\sigma$ by varying the length of the bounded edge $E$ of $C$. As $\sigma$ is one-dimensional all curves in $\sigma$ arise in that way. In particular, all interior curves of $\sigma$ do not have a contracted bounded edge and we can use the second description of $\phi_{1|2,3}$ from Equation \eqref{phi}. Let $\widetilde{C}$ be a curve in $\sigma$ and $\widetilde{E}$ be the deformed edge which originates from $E$, then we obtain:
\begin{align*}
\phi_{1|2,3}\left(\widetilde{C}\right)=&_{\eqref{phi}} \textrm{sum of lengths of edges that separate $1$ from $2,3$}\\
=&\textrm{length of $\widetilde{E}$} +\textrm{term that is constant on $\sigma$}, 
\end{align*}
On the other hand we obtain:
\begin{align*}
ft \left(\widetilde{C}\right)=& \textrm{length of $\widetilde{E}$}.
\end{align*}
Since $\widetilde{C}$ was an arbitrary curve in $\sigma$, we can conclude that $\phi_{1|2,3}|_{\sigma}+\textrm{constant}=ft|_{\sigma}$.
So the linear parts of $\phi_{1|2,3}|_{\sigma}$ and $ft|_{\sigma}$ coincide.
\end{proof}

Now we have to quantify the weight of curves as shown in \fref{specialexample}. As already mentioned such weights are given by determinantal expressions. 
A nice example of how to compute such an absolute determinant $\left|\det_{C}\left(\ldots\right)\right|$ for a curve $C$ is presented in \cite[Example 2.6]{franz}. However, here we use notions analogous to \cite[Notation 5.2]{markrau}. In particular, we consider evaluation maps 
\begin{align}
ev:X\rightarrow \mathbb{R}^{1} \times \left(\mathbb{R}^{2}\right)^{n}, C \mapsto \left(\left(ev_{1}\left(C\right)_{x}\right),\left(ev_{k}\left(C\right)\right)_{k \in N}\right),
\end{align} 
\begin{align}
\widetilde{ev}:X\rightarrow \left(\mathbb{R}^{2}\right)^{n+1}, C \mapsto \left(\left(ev_{1}\left(C\right)\right),\left(ev_{k}\left(C\right)\right)_{k \in N}\right),
\end{align} 
where the index $x$ at a point in $\mathbb{R}^{2}$ indicates its first coordinate. Later on we will also use the index $y$ to indicate a second coordinate.
So the evaluation map $\widetilde{ev}$ arises from $ev$ by replacing the line restriction at $x_{1}$ with a point restriction at $x_{1}$.

\begin{lemma}
\label{importantnewformula}
Let $Z=\left(ft^{*}\left(\lambda\right) \tau_{0}\left(1\right) \prod_{k\in N}\tau_{r_{k}}\left(2\right)\right)$ be a zero-dimensional intersection product and $X$be as above. Let $4 \in \Delta$. Moreover, let $C$ be an interior curve in a facet of $F$ as shown in \fref{specialexample}, such that $x_{4}$ is the unbounded edge adjacent to the growing bounded edge $E$ and $C$ maps to a large $\lambda$ under $ft$.
Then, the weight of $C$ in $Z$ is $\left|\det_{C}\left(ft \times ev\right)\right|=\left|\det_{C}\left( \widetilde{ev}\right)\right|$.
\end{lemma}

\begin{proof}
Let $C$ be a curve as above.
Let $l_{1}$ be the length of its growing bounded edge $E$ (see \fref{specialexample}).
We want to compute the weight of $C$ in $Z$. In so doing, we choose a marked end $x_{a}\neq x_{1}$ as the anchor vertex. Such a marked end must exist, because the number of marked ends is greater than three by Assumption \ref{notation}.
The position of $x_{a}$ is determined by the coordinates $\left(h\left(x_{a}\right)\right)_{x}$ and $\left(h\left(x_{a}\right)\right)_{y}$.
Say $E$ runs along the standard direction $-e_{1}$, then the position of the marked end $x_{1}$ is given as
\begin{align}
\label{eq:coord1}
h\left(x_{1}\right)=h\left(x_{a}\right)- l_{1} \cdot e_{1}+\underbrace{\ldots}_{\substack{\textrm{expression in the}\\ \textrm{other bounded lengths}}}.
\end{align}
The descriptions of the positions of all other marked ends do not involve the length $l_{1}$, because we chose $x_{a}\neq x_{1}$.

Now, by \cite[Lemma 2.4]{franz} we know that the weight of $C$ is given as the absolute value of the determinant $\left|\det_{C}\left(ft \times ev\right)\right|$.
We call the underlying matrix of this absolute determinant $H$.
Let the line $G_{1}$ be given by $\max\{y,c_{1}\}$. Thus, the pull-back of the evaluation map with plugged in coordinates from \eqref{eq:coord1} is of the form $\max\{\left(h\left(x_{a}\right)\right)_{y}+\ldots-c_{1},0\}$.
We assume that the points $P_{k}$ are given by functions of the form $\max\{x,p^{k}_{1}\}$ and $\max\{y,p^{k}_{2}\}$. 
Then, the pull-backs of the evaluation maps at the points are of the form $\max\{\left(h\left(x_{k}\right)\right)_{x}-p^{k}_{1},0\}$ and $\max\{\left(h\left(x_{k}\right)\right)_{y}-p^{k}_{2},0\}$. 
So we can determine the entries of all rows of $H$ that correspond to pull-backs of evaluations at the line $G_{1}$ and the points in the columns corresponding to $l_{1}$, $\left(h\left(x_{a}\right)\right)_{x}$ and $\left(h\left(x_{a}\right)\right)_{y}$. The only missing row is the row corresponding to the pull-back of the $\mathcal{M}_{0,4,trop}$-coordinate $\lambda$ under $ft$.
But, since $\lambda$ is large and there is no contracted bounded edge in a curve in this facet, the length $l_{1}$ must contribute to the length of $\lambda$. This causes an entry one at the column corresponding to $l_{1}$. 
All in all, the entries in $H$ are summarized in the table below.

\begin{center}
\begin{tabular}{c|| c c c c}
&&&& other bounded \\
 & $\left(h\left(x_{a}\right)\right)_{x}$ & $\left(h\left(x_{a}\right)\right)_{y}$ & $l_{1}$ & lengths\\

 \hline \hline

evaluation at $G_{1}$ & 0 & 1 & 0 & *\\

evaluation at $\left(P_{k}\right)_{x}$ for $k \in N$& 1 & 0 & 0 & *\\

evaluation at $\left(P_{k}\right)_{y}$ for $k \in N$ & 0 & 1 & 0 & *\\
$ft$ & 0 & 0 & 1 & * \\
\end{tabular}
\end{center}
\vspace{0.4cm}

We can permute rows and columns of $H$, because we are only interested in the absolute value of its determinant.
So we permute the rows and columns of $H$ in such a way that its first row is the row corresponding to $ft$ and that this row starts with a one.
Then, we can apply the determinantal formula for block matrices on the upper left 1 by 1 matrix block. This means we forget the row corresponding to $ft$ and the column corresponding to $l_{1}$. Now, applying the determinantal formula for block matrices backwards we can reintroduce a column corresponding to $l_{1}$ and introduce a row corresponding to $\left(P_{1}\right)_{x}$ for a new point $P_{1}$ which lies on the end pointing towards $-e_{1}$ of the line $G_{1}$. We end up with the matrix entries shown in the table below.

\begin{center}
\begin{tabular}{c|| c c c c}
&&&&other bounded\\
 & $l_{1}$ & $\left(h\left(x_{a}\right)\right)_{x}$ & $\left(h\left(x_{a}\right)\right)_{y}$  &  lengths \\
 \hline \hline

evaluation at $\left(P_{1}\right)_{x}$ & -1 & 1 & 0 & *\\

evaluation at $G_{1}=\left(P_{1}\right)_{y}$ & 0 & 0 & 1 & *\\

evaluation at $\left(P_{k}\right)_{x}$ for $k \in N$ &0 & 1 & 0 & *\\

evaluation at $\left(P_{k}\right)_{y}$ for $k \in N$ &0 & 0 & 1 & *\\ 

\end{tabular}
\end{center}
\vspace{0.4cm}

But this is exactly the matrix of the curve $C$ for the evaluation map $\widetilde{ev}$ which arises from $ev$ by replacing the line restriction at $x_{1}$ with a point restriction given by $P_{1}$.
This yields $\left|\det_{C}\left(ft \times ev\right)\right|=\left|\det_{C}\left( \widetilde{ev}\right)\right|$ and, hence, proves the claim.
\end{proof}

Finally we have all tools available to prove the main result of this article.

\begin{theorem}
\label{majortheorem}
It holds:
\begin{align*}
&\left\langle \tau_{1}\left(1\right) \prod_{k\in N}\tau_{r_{k}}\left(2\right)\right\rangle_{d}=\\
&\sum\left\langle \tau_{0}\left(\varepsilon\right)\tau_{0}\left(1\right)\prod_{k\in N \cap I}\tau_{r_{k}}\left(2\right)\right\rangle_{d_{1}}\cdot\left\langle \tau_{0}\left(\zeta\right) \prod_{k\in N\cap J}\tau_{r_{k}}\left(2\right)\right\rangle_{d_{2}}\\
&+3 \left\langle \tau_{0}\left(2\right) \prod_{k\in N}\tau_{r_{k}}\left(2\right)\right\rangle_{d}
\end{align*}
where the sum ranges over all $\varepsilon+\zeta=2$ with $\varepsilon,\zeta \in \mathbb{Z}_{\geq 0}$ and $d_{1}+d_{2}=d$ with $d_{1},d_{2} \in \mathbb{Z}_{\geq 0}$ and reducible partitions $I|J$ with $1\in I$ and $2,3\in J$.
\end{theorem}

\begin{proof}
By \eqref{eqhardterms2} we can write $\left\langle \tau_{1}\left(1\right) \prod_{k\in N}\tau_{r_{k}}\left(2\right)\right\rangle_{d}$ as 
\begin{align*}
&\left\langle\tau_{1}\left(1\right) \prod_{k\in N}\tau_{r_{k}}\left(2\right)\right\rangle_{d}=\\
&\sum_{\substack{I|J \textrm{ reducible }\\ 1 \in I, 2,3 \in J}} \sum_{\varepsilon+\zeta=2}\left\langle \tau_{0}\left(\varepsilon\right)\tau_{0}\left(1\right)\prod_{k\in N \cap I}\tau_{r_{k}}\left(2\right)\right\rangle_{d_{1}}\cdot\left\langle \tau_{0}\left(\zeta\right)\prod_{k\in N\cap J}\tau_{r_{k}}\left(2\right)\right\rangle_{d_{2}}\\
&+  \left\langle \phi_{1|2,3} \tau_{0}\left(1\right) \prod_{k\in N}\tau_{r_{k}}\left(2\right)\right\rangle_{d}.
\end{align*}
It remains to quantify $\left\langle \phi_{1|2,3} \tau_{0}\left(1\right) \prod_{k\in N}\tau_{r_{k}}\left(2\right)\right\rangle$. For reasons of clarity, we denote the set of facets of $F$ as $F^{\left(1\right)}$ and the set of curves, to wit vertices, of $F$ as $F^{\left(0\right)}$. Then by the definition of tropical divisors \cite[Definition 3.4]{allerau} it holds
\begin{align}
\label{nullteform}
\left\langle \phi_{1|2,3}  \tau_{0}\left(1\right) \prod_{k\in N}\tau_{r_{k}}\left(2\right)\right\rangle_{d}=\sum_{F^{\left(0\right)} \ni C < \sigma \in F^{\left(1\right)}} w_{F}\left(\sigma\right) \left(\phi_{1|2,3}\right)_{\sigma}\left(u_{\sigma/C}\right),
\end{align}
where $w_{F}\left(\sigma\right)$ is the weight of the facet $\sigma$ in $F$ and $u_{\sigma/C}$ is the primitive generator of the integer vector pointing from $C$ to $\sigma$.
For facets $\sigma \in F^{\left(1\right)}$ that contain two different vertices the two
corresponding primitive generators have opposite directions (see \cite[Lemma 1.4.4]{rau}). Thus the two respective summands in \eqref{nullteform} cancel out. The $\sigma \in F^{\left(1\right)}$ that do not contain two different vertices are unbounded polyhedra. 
We know from \cref{lemmanewdefor} that $\phi_{1|2,3}$ is bounded on all facets of $F$ which do not have interior curves like in \fref{specialexample}. But on unbounded polyhedra a rational function can only be bounded if it is constant, i.e. $\left(\phi_{1|2,3}\right)_{\sigma} \equiv 0$. So the summands in \eqref{nullteform} belonging to polyhedra which do not have interior curves like in \fref{specialexample} vanish. In fact, it holds:
\begin{align}
\label{ersteform}
\left\langle \phi_{1|2,3}  \tau_{0}\left(1\right) \prod_{k\in N}\tau_{r_{k}}\left(2\right)\right\rangle_{d}=\sum_{\substack{F^{\left(0\right)} \ni C < \sigma \in F^{\left(1\right)}:\\ \sigma \textrm{ has interior curve}\\ \textrm{like in \fref{specialexample}} }} w_{F}\left(\sigma\right) \left(\phi_{1|2,3}\right)_{\sigma}\left(u_{\sigma/C}\right).
\end{align}
Due to \lref{corollaryextrapoint} the linear parts $\left(\phi_{1|2,3}\right)_{\sigma}$ and $\left(ft\right)_{\sigma}$ coincide. Thus, it follows
\begin{align}
\label{zweiteform}
\left\langle \phi_{1|2,3}  \tau_{0}\left(1\right) \prod_{k\in N}\tau_{r_{k}}\left(2\right)\right\rangle_{d}=\sum_{\substack{F^{\left(0\right)} \ni C < \sigma \in F^{\left(1\right)}:\\ \sigma \textrm{ has interior curve}\\ \textrm{like in \fref{specialexample}} }} w_{F}\left(\sigma\right) \left(ft\right)_{\sigma}\left(u_{\sigma/C}\right).
\end{align}
Again, by \cite[Lemma 2.4]{franz} we can reformulate \eqref{zweiteform} and write (sticking to the notations from \lref{importantnewformula})
\begin{align}
\label{dritteform}
\left\langle \phi_{1|2,3}  \tau_{0}\left(1\right) \prod_{k\in N}\tau_{r_{k}}\left(2\right)\right\rangle_{d}=&\sum_{\substack{C \textrm{ looks like}\\
\textrm{ in \fref{specialexample}}}}  \left|\textrm{det}_{C}\left(ft \times ev\right)\right|.
\end{align}
Here, as well as in the next two equations, the curves $C$ over which we sum up live in the subcomplex containing the points of $X$ that map to the correct point $P_{k}$ for all $k \in N$ under the respective evaluation map and map to a fixed sufficiently large $\lambda$ under $ft$.
Now, we can divide the curves that look like in \fref{specialexample} into three classes depending along which standard direction the end $x_{4}$, which is the unbounded edge adjacent to the marked end $x_{1}$, runs. It holds:
\begin{align}
\left\langle \phi_{1|2,3}  \tau_{0}\left(1\right) \prod_{k\in N}\tau_{r_{k}}\left(2\right)\right\rangle_{d}=&\sum_{\substack{C \textrm{ looks like}\\
\textrm{ in \fref{specialexample}}}} \left|\textrm{det}_{C}\left(ft \times ev\right)\right|\notag\\
=&\sum_{\substack{C \textrm{ looks like in}\\
\textrm{\fref{specialexample} and}\\ \textrm{$x_{4}$ runs along $-e_{1}$}}}  \left|\textrm{det}_{C}\left(ft \times ev\right)\right|\notag\\
&+\sum_{\substack{C \textrm{ looks like in}\\
\textrm{\fref{specialexample} and}\\ \textrm{$x_{4}$ runs along $-e_{2}$}}}  \left|\textrm{det}_{C}\left(ft \times ev\right)\right|\notag\\
&+\sum_{\substack{C \textrm{ looks like in}\\
\textrm{\fref{specialexample} and}\\ \textrm{$x_{4}$ runs along $e_{1}+e_{2}$}}} \left|\textrm{det}_{C}\left(ft \times ev\right)\right|\notag.
\end{align}

Finally, from \lref{importantnewformula} we obtain
\begin{align}
\label{vierteform}
\left\langle \phi_{1|2,3}  \tau_{0}\left(1\right) \prod_{k\in N}\tau_{r_{k}}\left(2\right)\right\rangle_{d}=&\sum_{\substack{C \textrm{ looks like in}\\
\textrm{\fref{specialexample} and}\\ \textrm{$x_{4}$ runs along $-e_{1}$}}} \left|\textrm{det}_{C}\left( \widetilde{ev}'\right)\right|\notag\\
&+\sum_{\substack{C \textrm{ looks like in}\\
\textrm{\fref{specialexample} and}\\ \textrm{$x_{4}$ runs along $-e_{2}$}}} \left|\textrm{det}_{C}\left( \widetilde{ev}''\right)\right|\notag\\
&+\sum_{\substack{C \textrm{ looks like in}\\
\textrm{\fref{specialexample} and}\\ \textrm{$x_{4}$ runs along $e_{1}+e_{2}$}}} \left|\textrm{det}_{C}\left( \widetilde{ev}'''\right)\right|,
\end{align}
where, $\widetilde{ev}'$, $\widetilde{ev}''$ and $\widetilde{ev}'''$ arise from $ev$ by replacing the line restriction at $x_{1}$ with a point restriction given by new points $P_{1'}$, $P_{1''}$ and $P_{1'''}$ on $G_{1}$ (in the sense of \lref{importantnewformula}). The points will be chosen in a specific way as described later in the case of $P_{1'}$.
Now we observe that the sums in \eqref{vierteform} look almost like the number (see \cite[Theorem 5.3]{markrau}) $\left\langle\tau_{0}\left(2\right) \prod_{k\in N}\tau_{r_{k}}\left(2\right)\right\rangle=\sum_{C \in S} \left|\det_{C}\left( \widetilde{ev}\right)\right|$ where $S$ is the subcomplex of $X$ containing all points in $X$ that map to the correct point $P_{k}$ for $k \in N \cup\{1\}$ under the respective evaluation map. So it is natural to try to express each of the three sums via the invariant $\left\langle\tau_{0}\left(2\right) \prod_{k\in N}\tau_{r_{k}}\left(2\right)\right\rangle$. In so doing, we have to relate the index of summation of each of the sums in \eqref{vierteform} to the index of summation occurring in the sum expression of $\left\langle\tau_{0}\left(2\right) \prod_{k\in N}\tau_{r_{k}}\left(2\right)\right\rangle$.\\ 
Making this precise, we define three subcomplexes $S_{1}$, $S_{2}$ and $S_{3}$ of $X$ where $S_{1}$ contains all points in $X$ that map to the correct point $P_{k}$ for $k \in N \cup\{1'\}$ under the respective evaluation map, $S_{2}$ contains all points that map to $P_{k}$ for $k \in N \cup\{1''\}$ and $S_{3}$ contains all points that map to $P_{k}$ for $k \in N \cup\{1'''\}$. We want to show the following three equalities of (finite) sets:
 \begin{align}\{C \,|\, C \textrm{ looks like in \fref{specialexample} and $x_{4}$ runs along $-e_{1}$}\}=S_{1}\end{align}
 \begin{align}\{C \,|\, C \textrm{ looks like in \fref{specialexample} and $x_{4}$ runs along $-e_{2}$}\}=S_{2}\end{align}
 \begin{align}\{C \,|\, C \textrm{ looks like in \fref{specialexample} and $x_{4}$ runs along $e_{1}+e_{2}$}\}=S_{3}.\end{align} 
This then proves the claim, because by \eqref{vierteform} we could write \begin{align*}
\left\langle \phi_{1|2,3}  \tau_{0}\left(1\right) \prod_{k\in N}\tau_{r_{k}}\left(2\right)\right\rangle_{d}
&=3 \left\langle \tau_{0}\left(2\right) \prod_{k\in N}\tau_{r_{k}}\left(2\right)\right\rangle_{d}.
\end{align*}
So let us check the two inclusions for $S_{1}$. Subsequently, the argument can be carried out analogously for $S_{2}$ and $S_{3}$.\\
First of all, we assume that the line $G_{1}$ and the points $P_{k}$ for $k \in N \cup\{1'\}$ are organized in the following way:
\begin{itemize}
	\item the $x$- and $y$-coordinates of the root of the line $G_{1}$ and of the points $P_{k}$ for $k \in N$ are in an open interval $\left(-\epsilon, \epsilon\right)$ for some small real number $\epsilon>0$;
  \item the point $P_{1'}$ lies on the end of the line $G_{1}$ which points into the direction $-e_{1}$ and its $x$-coordinate is $-\lambda$ where $\lambda \gg 0$ is some large real number;
	\item (additionally) the point $P_{2}$ lies at the origin and the $x$-coordinate of $P_{3}$ lies in the interval $\left[0,\epsilon\right]$.
\end{itemize}
This means we assume that the point $P_{1'}$ lies "far away" from the other points $P_{k}$ for $k\in N$ (see \fref{1:3}). Nevertheless, we still assume that all the points are in general position. This is possible because the points in general position are dense in the set of all points (cf. Section \ref{definitions}).
Note that we are now in the situation of \cite[Lemma 2.2.1]{torchiani}.
As a consequence of this lemma we get the two desired inclusions.\\
"$\subseteq$": Let $C$ be a curve that looks like in \fref{specialexample} and its unbounded edge $x_{4}$ runs along $-e_{1}$. Particularly, the edge $x_{4}$ runs along the end of $G_{1}$ pointing towards $-e_{1}$ and it is incident to a unique vertex. This vertex has an $x$-coordinate greater than $-\lambda$, because $C$ maps to $\lambda$ under $ft$ by assumption and the points $P_{2}$ and $P_{3}$ have an $x$-coordinate in the interval $\left[0,\epsilon\right]$.
Furthermore, the point $P_{1'}$ lies on the end of $G_{1}$ pointing towards $-e_{1}$ and it is placed at an $x$-coordinate of $-\lambda$ by construction. Thus, the unbounded edge $x_{4}$ of $C$ has to meet $P_{1'}$. All in all, the curve $C$ maps to $P_{1'}$ and $P_{k}$ for $k \in N$ under the respective evaluation maps. So $C$ is in $S_{1}$.\\
"$\supseteq$": Let $C \in S_{1}$. We want to show that $C$ must look like in \fref{1:3}, i.e. especially like in \fref{specialexample}.
\begin{figure}[h]
  \centering
		\def\svgwidth{360pt}
  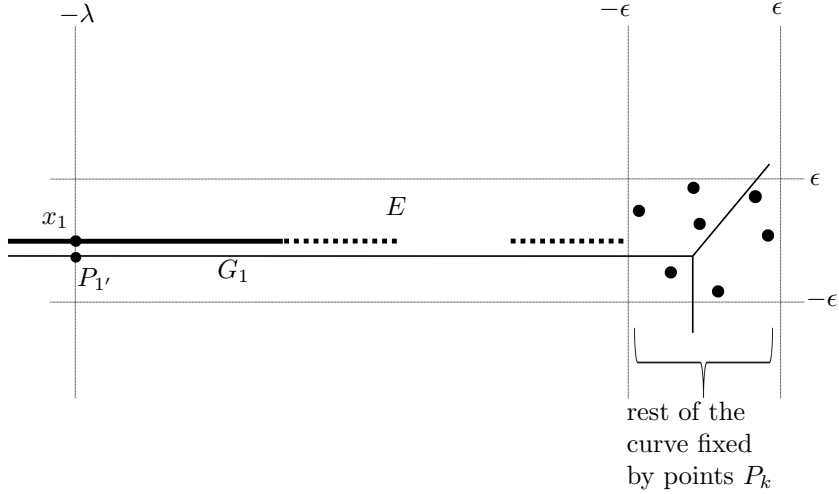
  \caption{Curve $C$ in $\left(\tau_{0}\left(2\right) \prod_{k\in N}\tau_{r_{k}}\left(2\right)\right)$ for special point configuration}
  \label{1:3}
\end{figure}
Therefore, assume that $C$ does not look like in \fref{1:3}. That means we assume that $x_{1}$ is not adjacent to an unbounded edge. By the balancing condition one of its adjacent bounded edges must be pointing leftwards and as this edge is bounded it is adjacent to a vertex with an $x$-coordinate smaller than $-\lambda$. But such a vertex does not exist, because all vertices of $C$ must lie in $R:=\{\left(x,y\right)\in \mathbb{R}^{2}| -\lambda \leq x\leq \epsilon, -\epsilon \leq y\leq \epsilon\}$ by \cite[Lemma 2.2.1]{torchiani}. So this case cannot occur. Hence, the curve $C \in S_{1}$ looks like in \fref{1:3} and the unbounded edge $x_{4}$, which is the unbounded edge adjacent to $x_{1}$, runs along $-e_{1}$.
\end{proof}

By \tref{majortheorem} we are able to compute all tropical gravitational descendants of the form $\left\langle\tau_{1}\left(1\right) \prod_{k\in N}\tau_{r_{k}}\left(2\right)\right\rangle_{d}$. Using in addition the Tropical String Equation \cite[Theorem 3.12]{raunew}, the Tropical Divisor Equation \cite[Theorem 3.17]{raunew} and the Tropical Dilaton Equation \cite[Theorem 3.13]{raunew} we can recursively compute all tropical gravitational descendants of the form 
\begin{align}
\left\langle\tau_{0}\left(0\right)^{l'}\tau_{1}\left(0\right)^{l''} \tau_{0}\left(1\right)^{m-1}\tau_{1}\left(1\right)\prod_{k\in N}\tau_{r_{k}}\left(2\right)\right\rangle_{d}.
\end{align}

\begin{remark}
\label{remarkfirst}
This means we are in particular able to compute the tropical counterparts of the so-called first descendant invariants with exactly one tangency condition $\tau_{1}\left(1\right)$. First descendant invariants are gravitational descendants which involve only psi-classes of power at most one. They play an important role in many classical enumerative questions because they are sufficient to study characteristic number problems (cf. \cite{grabkopand}).
\end{remark}

To close this article and to see the new recursion formula in action, we want to compute the invariant $\left\langle \tau_{1}\left(1\right)\tau_{1}\left(2\right)\tau_{1}\left(2\right)\right\rangle_{2}$. In fact, it holds:

\begin{align}
\label{doublecheck} \left\langle \tau_{1}\left(1\right)\tau_{1}\left(2\right)\tau_{1}\left(2\right)\right\rangle_{2}
=_{\ref{majortheorem}}& \left( \sum \left\langle \tau_{0}\left(\varepsilon\right) \tau_{0}\left(1\right)\right\rangle_{d_{1}} \cdot\left\langle \tau_{0}\left(\zeta\right) \tau_{1}\left(2\right) \tau_{1}\left(2\right)\right\rangle_{d_{2}}\right.\notag\\
&\left.+3\underbrace{\left\langle\tau_{0}\left(2\right)\tau_{1}\left(2\right)\tau_{1}\left(2\right)\right\rangle_{2}}
_{=_{\cite{markrau}} 1}  \right)\notag\\
=& \underbrace{\left\langle \tau_{0}\left(0\right) \tau_{0}\left(1\right)\right\rangle_{0} \cdot\left\langle \tau_{0}\left(2\right) \tau_{1}\left(2\right) \tau_{1}\left(2\right)\right\rangle_{2}}_{=0} + 3 \notag \\
=& 3.
\end{align}

The numerical value of the algebro-geometric version $\left\langle\tau_{1}\left(1\right)\tau_{1}\left(2\right)\tau_{1}\left(2\right)\right\rangle^{alg}_{2}$ of this invariant equals 0 and can be easily computed via standard techniques (e.g. see \cite{kock}). So this is an example where the classical algebro-geometric and the tropical number do not coincide as mentioned in the introduction.

\noindent {Falko Gauss, Lehrstuhl für Mathematik VI, Institut für Mathematik, Universität Mannheim,
 A5, 6, 68131, Mannheim, Germany}\\
\emph{E-mail address:} gauss@math.uni-mannheim.de

\end{document}